\documentclass{article}%
\usepackage{amssymb}
\usepackage{amsfonts}
\usepackage{amsmath}
\usepackage[colorlinks=true, pdfstartview=FitV, linkcolor=blue,
citecolor=blue, urlcolor=blue]{hyperref}
\usepackage{graphicx,color}
\usepackage{graphicx}%
\providecommand{\U}[1]{\protect\rule{.1in}{.1in}}
\newtheorem{theorem}{Theorem}

\newtheorem{corollary}[theorem]{Corollary}

\newtheorem{definition}[theorem]{Definition}

\newtheorem{lemma}[theorem]{Lemma}

\newtheorem{proposition}[theorem]{Proposition}
\newtheorem{remark}[theorem]{Remark}

\newenvironment{proof}[1][Proof]{\noindent\textbf{#1.} }{\ \rule{0.5em}{0.5em}}
\begin{document}

\date{}
\title{Character analogues of certain Hardy-Berndt sums}
\author{M\"{u}m\"{u}n Can and Veli Kurt\\Department of Mathematics, Akdeniz University, 07058-Antalya, Turkey \\e-mails: mcan@akdeniz.edu.tr, vkurt@akdeniz.edu.tr}
\maketitle

\begin{abstract}
In this paper we consider transformation formulas for
\[
B\left(  z,s:\chi\right)  =\sum\limits_{m=1}^{\infty}\sum\limits_{n=0}%
^{\infty}\chi(m)\chi(2n+1)\left(  2n+1\right)  ^{s-1}e^{\pi im(2n+1)z/k}.
\]
We derive reciprocity theorems for the sums arising in these transformation
formulas and investigate certain properties of them. With the help of the
character analogues of the Euler--Maclaurin summation formula we establish
integral representations for the Hardy-Berndt character sums $s_{3,p}\left(
d,c:\chi\right)  $ and $s_{4,p}\left(  d,c:\chi\right)  $.



\end{abstract}

\section{Introduction}

Berndt \cite{b6} and Goldberg \cite{g} derived transformation formulas for the
logarithms of the classical theta functions $\theta_{j}(z),$ $j=2,3,4$\ . In
these formulas six different arithmetic sums arise that are known as Berndt's
arithmetic sums or Hardy sums. Goldberg \cite{g} show that these sums also
arise in the theory of $r_{m}(n)$, the number of representations of $n$ as a
sum of $m$ integral squares and in the study of the Fourier coefficients of
the reciprocals of $\theta_{j}(z),$ $j=2,3,4$. Three of of these sums, which
we call Hardy-Berndt sums, are defined for $c>0$ by
\[
S(d,c)=\sum\limits_{n=1}^{c-1}\left(  -1\right)  ^{n+1+\left[  dn/c\right]
},\text{ }s_{3}(d,c)=\sum\limits_{n=1}^{c-1}\left(  -1\right)  ^{n}%
\overline{B}_{1}\left(  \frac{dn}{c}\right)  ,\text{ }s_{4}(d,c)=\sum
\limits_{n=1}^{c-1}\left(  -1\right)  ^{\left[  dn/c\right]  },
\]
where $\overline{B}_{p}\left(  x\right)  $ are the Bernoulli functions (see
Section 2) and $[x]$ denotes the greatest integer not exceeding $x$.

Analogous to Dedekind sums these sums also obey reciprocity formulas. For
instance, for coprime positive integers $d$ and $c$ we have \cite{b6,g}:%
\begin{align}
S(d,c)+S(c,d)  &  =1,\text{ \ \ if }d+c\text{ is odd,}\label{8}\\
2s_{3}(d,c)-s_{4}(c,d)  &  =1-\frac{d}{c},\text{ \ if }c\text{ is odd}
\label{9}%
\end{align}
and \cite{m4}%
\[
s_{4}(c,d)+s_{4}(d,c)\equiv-1+cd\text{ }\left(  \text{mod }8\right)  .
\]
We note that various properties of these sums have been investigated by many
authors, see \ \cite{a3,b6,bg,c1,g,m1,m2,m4,pt,sy,s,zw}, and several
generalizations have been studied in \cite{cck1,c2,dc,lz,m3}.

A character analogue of classical Dedekind sum, called as Dedekind character
sum, appears in the transformation formula of\ a generalized Eisenstein series
$G\left(  z,s:\chi:r_{1},r_{2}\right)  $ (see (\ref{7}) below) associated to a
non-principle primitive character $\chi$ of modulus $k$ defined by Berndt in
\cite{b2}. This sum is defined by%
\[
s\left(  d,c:\chi\right)  =\sum\limits_{n=1}^{ck}\chi\left(  n\right)
\overline{B}_{1,\chi}\left(  \frac{dn}{c}\right)  \overline{B}_{1}\left(
\frac{n}{ck}\right)
\]
and possesses the reciprocity formula%
\[
s\left(  c,d:\chi\right)  +s\left(  d,c:\overline{\chi}\right)  =B_{1,\chi
}B_{1,\overline{\chi}},
\]
whenever $d$ and $c$ are coprime positive integers, and either $c$ or
$d\equiv0\left(  \operatorname{mod}k\right)  $ (\cite{b2}). Here $\overline
{B}_{p,\chi}\left(  x\right)  $ are the generalized Bernoulli functions (see
(\ref{3}) below) and $B_{p,\chi}=\overline{B}_{p,\chi}\left(  0\right)  $. The
sum $s\left(  d,c:\chi\right)  $ is generalized by%
\[
s_{p}\left(  d,c:\chi\right)  =\sum\limits_{n=1}^{ck}\chi\left(  n\right)
\overline{B}_{p,\chi}\left(  \frac{dn}{c}\right)  \overline{B}_{1}\left(
\frac{n}{ck}\right)
\]
and corresponding reciprocity formula is established (\cite{cck}).

To the writers' knowledge generalizations of Hardy-Berndt sums, in the sense
of $s_{p}\left(  d,c:\chi\right)  ,$ have not been studied. To introduce such
generalization, originated by Berndt's paper \cite{b2} and the fact
\[
\log\theta_{4}\left(  z\right)  =-2\sum\limits_{m=1}^{\infty}\sum
\limits_{n=0}^{\infty}\frac{1}{2n+1}e^{\pi im(2n+1)z},
\]
we set the function $B\left(  z,s:\chi\right)  $\ to be
\[
B\left(  z,s:\chi\right)  =\sum\limits_{m=1}^{\infty}\sum\limits_{n=0}%
^{\infty}\chi(m)\chi(2n+1)\left(  2n+1\right)  ^{s-1}e^{\pi im(2n+1)z/k}%
\]
for $Im\left(  z\right)  >0$\ and for all $s.$\ 

The objective of this paper is to obtain transformation formulas for $B\left(
z,s;\chi\right)  $\ and investigate certain properties of the sums that arise
in these formulas. These are generalizations, containing characters and
generalized Bernoulli functions, of the sums $S(d,c),$\ $s_{3}(d,c),$%
\ $s_{4}(d,c)$\ and the sums considered in \cite{cck1}. We will show that
these sums satisfy reciprocity formulas in the sense of (\ref{8}) and
(\ref{9}). We also give integral representations for them.

A brief plan of the paper is as follows: Section 2 is the preliminary section
where we give definitions and terminology needed. In Section 3 we state and
prove main theorems concerning transformation formulas for $B\left(
z,s;\chi\right)  $. In Section 4 we give definitions of the character
analogues of Hardy-Berndt sums and prove corresponding reciprocity theorems.
We present some additional results in Section 5, in particular we derive
several interesting formulas relating these sums by employing fixed points of
a modular transformation. In the final section we apply the character
analogues of Euler--Maclaurin summation formula\ to give an alternative proof
of one of the reciprocity formula.

\section{Preliminaries}

Throughout this paper $\chi$ denotes a non-principal primitive character of
modulus $k$. The letter $p$ always denotes positive integer. We use the
modular transformation $\left(  az+b\right)  /\left(  cz+d\right)  $ where
$a,$ $b,$ $c$ and $d$ are integers with $ad-bc=1$ and $c>0$. The upper
half-plane $\left\{  x+iy:y>0\right\}  $ will be denoted by $\mathbb{H}$ and
the upper quarter-plane $\left\{  x+iy:x>-\frac{d}{c}\text{, }y>0\right\}  $
by $\mathbb{K}$. We use the notation $\left\{  x\right\}  $ for the fractional
part of $x.$ Unless otherwise stated, we assume that the branch of the
argument is defined by $-\pi\leq$ arg $z<\pi$.

The Bernoulli polynomials $B_{n}(x)$ are defined by means of the generating
function
\[
\frac{te^{xt}}{e^{t}-1}=\sum\limits_{n=0}^{\infty}B_{n}(x)\frac{t^{n}}{n!}%
\]
and $B_{n}(0)=B_{n}$ are the Bernoulli numbers with $B_{0}=1,$ $B_{1}=-1/2$
and $B_{2n-1}\left(  \frac{1}{2}\right)  =B_{2n+1}=0$ for $n\geq1.$

The Bernoulli functions $\overline{B}_{n}\left(  x\right)  $ are defined by
\begin{align*}
\overline{B}_{n}\left(  x\right)   &  =\text{ \ \ }B_{n}\left(  \left\{
x\right\}  \right)  ,\text{ \ \ \ }n>1,\\
\overline{B}_{1}(x)  &  =\left\{
\begin{array}
[c]{cl}%
0, & \hspace{-0.08in}x\text{ integer,}\\
B_{1}(\left\{  x\right\}  ), & \hspace{-0.08in}\text{otherwise}%
\end{array}
\right.
\end{align*}
and satisfy Raabe theorem for any $x$
\begin{equation}
\sum\limits_{j=0}^{r-1}\overline{B}_{n}\left(  x+\frac{j}{r}\right)
=r^{1-n}\overline{B}_{n}\left(  rx\right)  . \label{1}%
\end{equation}

$\overline{E}_{n}\left(  x\right)  $\ are the Euler functions and we are only
interested in the following property (\cite[Eq. (4.5)]{cl})%
\begin{equation}
r^{n-1}\sum\limits_{j=0}^{r-1}\left(  -1\right)  ^{j}\overline{B}_{n}\left(
\frac{x+j}{r}\right)  =-\frac{n}{2}\overline{E}_{n-1}\left(  x\right)
\label{11}%
\end{equation}
for even $r$ and any $x.$

$\overline{B}_{n,\chi}\left(  x\right)  $ are the generalized Bernoulli
functions defined by Berndt \cite{b5}.\ We will often use the following
property that can confer as a definition
\begin{equation}
\overline{B}_{n,\chi}\left(  x\right)  =k^{n-1}\sum_{j=0}^{k-1}\overline{\chi
}\left(  j\right)  \overline{B}_{n}\left(  \frac{j+x}{k}\right)  ,\ n\geq1.
\label{3}%
\end{equation}

The Gauss sum $G(z,\chi)$ is defined by\
\[
G(z,\chi)=\sum\limits_{m=0}^{k-1}\chi(m)e^{2\pi imz/k}.
\]
We put $G(1,\chi)=G(\chi)$. If $n$ is an integer, then \cite[p. 168]{a2}
\[
G(n,\chi)=\overline{\chi}(n)G(\chi).
\]

Let $r_{1}$ and $r_{2}$ be arbitrary real numbers. For $z\in\mathbb{H}$ and
$Re\left(  s\right)  >2,$ Berndt \cite{b2} defines $G\left(  z,s:\chi
:r_{1},r_{2}\right)  $ as%
\[
G\left(  z,s:\chi:r_{1},r_{2}\right)  =\sum\limits_{m,n=-\infty}^{\infty
}\ \hspace{-0.19in}^{^{\prime}}\ \frac{\chi(m)\overline{\chi}(n)}{\left(
\left(  m+r_{1}\right)  z+n+r_{2}\right)  ^{s}},
\]
where the dash means that the possible pair $m=-r_{1},$ $n=-r_{2}$ is omitted
from the summation. Extending the definition of $\chi$ to the set of all real
numbers by defining $\chi\left(  r\right)  =0$ if $r$ is not an integer, it is
shown that
\begin{align}
G\left(  z,s:\chi:r_{1},r_{2}\right)   &  =\frac{G(\overline{\chi})\left(
-2\pi i/k\right)  ^{s}}{\Gamma\left(  s\right)  }\left\{  A\left(
z,s:\chi:r_{1},r_{2}\right)  +e^{\pi is}A\left(  z,s:\chi:-r_{1}%
,-r_{2}\right)  \right\} \nonumber\\
&  \quad+\chi\left(  -r_{1}\right)  \left\{  L\left(  s,\overline{\chi}%
,r_{2}\right)  +\chi\left(  -1\right)  e^{\pi is}L\left(  s,\overline{\chi
},-r_{2}\right)  \right\}  , \label{7}%
\end{align}
where%
\[
L\left(  s,\chi,\alpha\right)  =\sum\limits_{m>-a}^{\infty}\chi(m)\left(
m+\alpha\right)  ^{s},\ \alpha\text{ real and }Re(s)>1
\]
and%
\[
A\left(  z,s:\chi:r_{1},r_{2}\right)  =\sum\limits_{m>-r_{1}}^{\infty}%
\chi(m)\sum\limits_{n=1}^{\infty}\chi(n)n^{s-1}e^{2\pi in\left(  \left(
m+r_{1}\right)  z+r_{2}\right)  /k}.
\]
For $r_{1}=r_{2}=0$ we will use the notations $G\left(  z,s:\chi\right)
=G\left(  z,s:\chi:0,0\right)  $ and $A\left(  z,s:\chi\right)  =A\left(
z,s:\chi:0,0\right)  .$ Also, for $r_{1}=r_{2}=0,$ (\ref{7}) reduces to
\[
\Gamma\left(  s\right)  G\left(  z,s:\chi\right)  =G(\overline{\chi})\left(
-\frac{2\pi i}{k}\right)  ^{s}H\left(  z,s:\chi\right)
\]
where $H\left(  z,s:\chi\right)  =\left(  1+e^{\pi is}\right)  A\left(
z,s:\chi\right)  .$

The following lemma due to Lewittes \cite[Lemma 1]{lew}.

\begin{lemma}
\label{le1} Let $A,$ $B,$ $C$ and $D$ be real with $A$ and $B$ not both zero
and $C>0$. Then for $z\in\mathbb{H},$%
\[
arg\left(  \left(  Az+B\right)  /\left(  Cz+D\right)  \right)  =arg\left(
Az+B\right)  -arg\left(  Cz+D\right)  +2\pi l,
\]
where $l$ is independent of $z$ and $l=\left\{
\begin{array}
[c]{ll}%
1, & A\leq0\text{ and }AD-BC>0,\\
0, & \text{otherwise.}%
\end{array}
\right.  $
\end{lemma}

In accordance with the subject of this study we present Berndt's
transformation formulas for $r_{1}=r_{2}=0$ (see \cite[Theorem 2]{m3} for a generalization).

\begin{theorem}
\label{tk1}\cite[Theorem 2]{b2} Let $Tz=\left(  az+b\right)  /\left(
cz+d\right)  $.\ Suppose first that $a\equiv d\equiv0(mod$ $k).$ Then for
$z\in\mathbb{K}$ and all $s,$
\begin{align*}
&  \left(  cz+d\right)  ^{-s}\Gamma\left(  s\right)  G\left(  Tz,s:\chi\right)
\\
&  \quad=\overline{\chi}(b)\chi(c)\Gamma\left(  s\right)  G\left(
z,s:\overline{\chi}\right)  +\overline{\chi}(b)\chi(c)e^{-\pi is}%
\sum\limits_{j=1}^{c}\sum\limits_{\mu=0}^{k-1}\sum\limits_{\nu=0}%
^{k-1}\overline{\chi}\left(  \mu c+j\right)  \chi\left(  \left[  \tfrac{dj}%
{c}\right]  -\nu\right)  f(z,s:c,d),
\end{align*}
where
\begin{equation}
f(z,s:c,d)=\int\limits_{C}\frac{e^{-\left(  \mu c+j\right)  \left(
cz+d\right)  u/c}}{e^{-\left(  cz+d\right)  ku}-1}\frac{e^{\left(
\nu+\left\{  dj/c\right\}  \right)  u}}{e^{ku}-1}u^{s-1}du, \label{51}%
\end{equation}
where $C$ is a loop beginning at $+\infty,$ proceeding in the upper
half-plane, encircling the origin in the positive direction so that $u=0$ is
the only zero of $\left(  e^{-\left(  cz+d\right)  ku}-1\right)  \left(
e^{ku}-1\right)  $ lying \textquotedblleft inside" the loop, and then
returning to $+\infty$ in the lower half-plane. Here we choose the branch of
$u^{s}$ with $0<arg$ $u<2\pi.$

Secondly, if $b\equiv c\equiv0(mod$ $k),$ we have for $z\in\mathbb{K}$ and all
$s,$
\begin{align*}
&  \left(  cz+d\right)  ^{-s}\Gamma\left(  s\right)  G\left(  Tz,s:\chi\right)
\\
&  \quad=\overline{\chi}(a)\chi(d)\Gamma\left(  s\right)  G\left(
z,s:\chi\right)  +\overline{\chi}(a)\chi(d)e^{-\pi is}\sum\limits_{j=1}%
^{c}\sum\limits_{\mu=0}^{k-1}\sum\limits_{\nu=0}^{k-1}\chi\left(  j\right)
\overline{\chi}\left(  \left[  \tfrac{dj}{c}\right]  +d\mu-\nu\right)
f(z,s:c,d).
\end{align*}

\end{theorem}

\section{Transformation Formulas}

In the sequel, unless otherwise stated, we assume that $k$\ is odd.

Put $B^{\prime}\left(  z,s:\chi\right)  =\left(  1+e^{\pi is}\right)  B\left(
z,s:\chi\right)  .$ We then use the relation%
\begin{equation}
B\left(  z,s:\chi\right)  =A\left(  \frac{z}{2},s:\chi\right)  -\chi\left(
2\right)  2^{s-1}A\left(  z,s:\chi\right)  \label{4}%
\end{equation}
in order to achieve transformation formulas for $B^{\prime}\left(
z,s:\chi\right)  .$ 

\begin{theorem}
\label{tk2}Let $Tz=\left(  az+b\right)  /\left(  cz+d\right)  $ with $b$ is
even. If $a\equiv d\equiv0(mod$ $k),$ then for $z\in\mathbb{K}$ and all $s$
\begin{align}
&  \left(  cz+d\right)  ^{-s}G\left(  \overline{\chi}\right)  B^{\prime
}\left(  Tz,s:\chi\right)  =\overline{\chi}\left(  \frac{b}{2}\right)
\chi(2c)G(\chi)B^{\prime}\left(  z,s:\overline{\chi}\right) \nonumber\\
&  +\overline{\chi}\left(  \frac{b}{2}\right)  \chi(2c)\left(  -\frac{k}{2\pi
i}\right)  ^{s}e^{-\pi is}\sum\limits_{\mu=0}^{k-1}\sum\limits_{\nu=0}%
^{k-1}\left(  \sum\limits_{j=1}^{2c}\overline{\chi}\left(  2c\mu+j\right)
\chi\left(  \left[  \tfrac{dj}{2c}\right]  -\nu\right)  f\left(  \frac{z}%
{2},s:2c,d\right)  \right. \nonumber\\
&  \left.  -\overline{\chi}(2)2^{s-1}\sum\limits_{j=1}^{c}\overline{\chi
}\left(  \mu c+j\right)  \chi\left(  \left[  \tfrac{dj}{c}\right]
-\nu\right)  f(z,s:c,d)\right)  . \label{60}%
\end{align}
If $b\equiv c\equiv0(mod$ $k),$ then for $z\in\mathbb{K}$ and all $s$
\begin{align}
&  \left(  cz+d\right)  ^{-s}G\left(  \overline{\chi}\right)  B^{\prime
}\left(  Tz,s:\chi\right)  =\overline{\chi}(a)\chi(d)G(\overline{\chi
})B^{\prime}\left(  z,s:\chi\right) \nonumber\\
&  +\overline{\chi}(a)\chi(d)\left(  -\tfrac{k}{2\pi i}\right)  ^{s}e^{-\pi
is}\sum\limits_{\mu=0}^{k-1}\sum\limits_{\nu=0}^{k-1}\left(  \sum
\limits_{j=1}^{2c}\chi\left(  j\right)  \overline{\chi}\left(  \left[
\tfrac{\left(  d\right)  j}{2c}\right]  +d\mu-\nu\right)  f\left(  \frac{z}%
{2},s:2c,d\right)  \right. \nonumber\\
&  \left.  -\chi(2)2^{s-1}\sum\limits_{j=1}^{c}\chi\left(  j\right)
\overline{\chi}\left(  \left[  \tfrac{dj}{c}\right]  +d\mu-\nu\right)
f(z,s:c,d)\right)  , \label{61}%
\end{align}
where $f(\frac{z}{2},s:2c,d)$ and $f(z,s:c,d)$ are given by (\ref{51}).
\end{theorem}

\begin{proof}
Let $b$ be even and consider $Uz=\left(  az+\frac{b}{2}\right)  /\left(
2cz+d\right)  $. Since $U\left(  z/2\right)  =\frac{1}{2}T(z)$ we have by
(\ref{4})
\[
\left(  cz+d\right)  ^{-s}B^{\prime}(Tz,s:\chi)=\left(  2c\frac{z}%
{2}+d\right)  ^{-s}H\left(  U\left(  z/2\right)  ,s:\chi\right)  -\chi\left(
2\right)  2^{s-1}\left(  cz+d\right)  ^{-s}H(Tz,s:\chi)
\]
Applying Theorem \ref{tk1} to the right-hand side of equality we get the
desired results.
\end{proof}

This theorem may be simplified for nonpositive integer value of $s$ as in the following.

\begin{corollary}
\label{sk2}Let $p$ be odd and $Tz=\left(  az+b\right)  /\left(  cz+d\right)  $
with $b$ is even. If $a\equiv d\equiv0(mod$ $k),$ then for $z\in\mathbb{H}$
\begin{equation}
\left(  cz+d\right)  ^{p-1}G\left(  \overline{\chi}\right)  B^{\prime}\left(
Tz,1-p:\chi\right)  =\overline{\chi}(\frac{b}{2})\chi(2c)\left(
G(\chi)B^{\prime}\left(  z,1-p:\overline{\chi}\right)  +\frac{\left(  2\pi
i\right)  ^{p}\chi(-1)}{\left(  p+1\right)  !}g_{1}(c,d;z,p;\overline{\chi
})\right)  . \label{53a}%
\end{equation}
If $b\equiv c\equiv0(mod$ $k),$ then for $z\in\mathbb{H}$
\begin{align}
&  \hspace{-0.3in}\left(  cz+d\right)  ^{p-1}G\left(  \overline{\chi}\right)
B^{\prime}\left(  Tz,1-p:\chi\right) \nonumber\\
&  \ \hspace{-0.3in}=\overline{\chi}(a)\chi(d)\left(  G(\overline{\chi
})B^{\prime}\left(  z,1-p:\chi\right)  +\frac{\left(  2\pi i\right)
^{p}\overline{\chi}(-1)}{\left(  p+1\right)  !}g_{1}(c,d;z,p;\chi)\right)  ,
\label{62a}%
\end{align}
where
\begin{align}
&  g_{1}(c,d;z,p;\chi)\nonumber\\
&  =-\sum\limits_{m=1}^{p}\binom{p+1}{m}\left(  -\left(  cz+d\right)  \right)
^{m-1}k^{m-p}\frac{m}{2^{m}}\sum\limits_{n=1}^{ck}\chi(n)\overline
{B}_{p+1-m,\chi}\left(  \frac{dn}{2c}\right)  \overline{E}_{m-1}\left(
\frac{n}{ck}\right)  . \label{54a}%
\end{align}

\end{corollary}

\begin{proof}
For $s=1-p$ in Theorem \ref{tk2}, by residue theorem, we have%
\begin{equation}
f(z,1-p:c,d)=\frac{2\pi ik^{p-1}}{\left(  p+1\right)  !}\sum_{m=0}^{p+1}%
\binom{p+1}{m}\left(  -\left(  cz+d\right)  \right)  ^{m-1}B_{p+1-m}\left(
\frac{\nu+\left\{  dj/c\right\}  }{k}\right)  B_{m}\left(  \frac{\mu c+j}%
{ck}\right)  . \label{63}%
\end{equation}
Let $b$ be even and $a\equiv d\equiv0(mod$ $k).$ Substituting (\ref{63}) in
(\ref{60}) and by the fact that the sum over $\mu$ is zero for $m=0$ and the
sum over $\nu$ is zero for $m=p+1$, we have
\begin{align}
&  \left(  cz+d\right)  ^{p-1}G\left(  \overline{\chi}\right)  B^{\prime
}\left(  Vz,1-p:\chi\right)  =\overline{\chi}\left(  \frac{b}{2}\right)
\chi(2c)G(\chi)B^{\prime}\left(  z,1-p:\overline{\chi}\right) \nonumber\\
&  \quad+\overline{\chi}\left(  \frac{b}{2}\right)  \chi(2c)\frac{\left(  2\pi
i\right)  ^{p}}{\left(  p+1\right)  !}\sum_{m=1}^{p}\binom{p+1}{m}\left(
-\left(  cz+d\right)  \right)  ^{m-1}\nonumber\\
&  \quad\times\left\{  \sum\limits_{\mu=0}^{k-1}\sum\limits_{\nu=0}^{k-1}%
\sum\limits_{j=1}^{2c}\overline{\chi}\left(  2\mu c+j\right)  \chi\left(
\left[  \tfrac{dj}{2c}\right]  -\nu\right)  B_{p+1-m}\left(  \frac
{\nu+\left\{  dj/2c\right\}  }{k}\right)  B_{m}\left(  \frac{2\mu c+j}%
{2ck}\right)  \right. \label{64}\\
&  \qquad\left.  -\frac{\overline{\chi}(2)}{2^{p}}\sum\limits_{\mu=0}%
^{k-1}\sum\limits_{\nu=0}^{k-1}\sum\limits_{j=1}^{c}\overline{\chi}\left(  \mu
c+j\right)  \chi\left(  \left[  \tfrac{dj}{c}\right]  -\nu\right)
B_{p+1-m}\left(  \frac{\nu+\left\{  dj/c\right\}  }{k}\right)  B_{m}\left(
\frac{\mu c+j}{ck}\right)  \right\}  . \label{65}%
\end{align}

We first evaluate the triple sum in (\ref{64}). Observe that the triple sum is
unchanged if we replace $B_{p+1-m}\left(  \frac{\nu+\left\{  dj/2c\right\}
}{k}\right)  $ by $\overline{B}_{p+1-m}\left(  \frac{\nu+\left\{
dj/2c\right\}  }{k}\right)  $ since $B_{p+1-m}\left(  \frac{\nu+\left\{
dj/2c\right\}  }{k}\right)  =\overline{B}_{p+1-m}\left(  \frac{\nu+\left\{
dj/2c\right\}  }{k}\right)  $ when $0<\frac{\nu+\left\{  dj/2c\right\}  }%
{k}<1,$ and $\chi\left(  d\right)  =0$ when $\frac{\nu+\left\{  dj/2c\right\}
}{k}=0$ ($d\equiv0(mod$ $k)$)$.$ By the same reason the triple sum is
unchanged when $B_{m}\left(  \frac{2\mu c+j}{2ck}\right)  $ is replaced by
$\overline{B}_{m}\left(  \frac{2\mu c+j}{2ck}\right)  .$ By (\ref{3}),
\begin{align*}
&  \sum\limits_{\mu=0}^{k-1}\sum\limits_{j=1}^{2c}\overline{\chi}\left(  2\mu
c+j\right)  \overline{B}_{m}\left(  \tfrac{2\mu c+j}{2ck}\right)
\sum\limits_{\nu=0}^{k-1}\chi\left(  \left[  \tfrac{dj}{2c}\right]
-\nu\right)  \overline{B}_{p+1-m}\left(  \tfrac{\nu+\frac{dj}{2c}-\left[
\frac{dj}{2c}\right]  }{k}\right) \\
&  \quad=\chi\left(  -1\right)  k^{m-p}\sum\limits_{\mu=0}^{k-1}%
\sum\limits_{j=1}^{2c}\overline{\chi}\left(  2\mu c+j\right)  \overline
{B}_{p+1-m,\overline{\chi}}\left(  \frac{dj}{2c}\right)  \overline{B}%
_{m}\left(  \frac{2\mu c+j}{2ck}\right)  .
\end{align*}
Put $2\mu c+j=n$, where $1\leq n\leq2ck$ in the right side of equality above.
Since $\overline{B}_{p,\chi}\left(  -x\right)  =\left(  -1\right)
^{p}\overline{\chi}\left(  -1\right)  \overline{B}_{p,\chi}\left(  x\right)  $
and $\overline{B}_{p,\overline{\chi}}\left(  x+k\right)  =\overline
{B}_{p,\overline{\chi}}\left(  x\right)  ,$ (\ref{64}) becomes%
\begin{align}
&  \chi\left(  -1\right)  k^{m-p}\sum\limits_{n=1}^{2ck}\overline{\chi}\left(
n\right)  \overline{B}_{p+1-m,\overline{\chi}}\left(  \frac{dn}{2c}\right)
\overline{B}_{m}\left(  \frac{n}{2ck}\right) \label{66a}\\
&  \quad=\chi\left(  -1\right)  k^{m-p}\sum\limits_{n=1}^{ck}\overline{\chi
}\left(  n\right)  \overline{B}_{p+1-m,\overline{\chi}}\left(  \frac{dn}%
{2c}\right)  2\overline{B}_{m}\left(  \frac{n}{2ck}\right)  . \label{66}%
\end{align}
Secondly, (\ref{65}) can be evaluated as
\begin{equation}
\chi\left(  -1\right)  k^{m-p}\overline{\chi}(2)2^{-p}\sum\limits_{n=1}%
^{ck}\overline{\chi}\left(  n\right)  \overline{B}_{p+1-m,\overline{\chi}%
}\left(  \frac{dn}{c}\right)  \overline{B}_{m}\left(  \frac{n}{ck}\right)  .
\label{67}%
\end{equation}
It can be seen from (\ref{3}) and (\ref{1}) that
\begin{equation}
\chi(r)r^{1-m}\overline{B}_{m,\chi}\left(  rx\right)  =\sum_{j=0}%
^{r-1}\overline{B}_{m,\chi}\left(  x+\frac{jk}{r}\right)  \label{68}%
\end{equation}
for $(r,k)=1.$ Thus, in the light of (\ref{68}), (\ref{67}) becomes
\begin{equation}
\chi\left(  -1\right)  k^{m-p}\sum\limits_{n=1}^{ck}\overline{\chi}\left(
n\right)  \overline{B}_{p+1-m,\overline{\chi}}\left(  \frac{dn}{2c}\right)
2^{1-m}\overline{B}_{m}\left(  \frac{n}{ck}\right)  . \label{69}%
\end{equation}
Then, utilizing (\ref{1}) and (\ref{11}), (\ref{66}) and (\ref{69}) yield
(\ref{53a}) for $z\in\mathbb{K}$. Since the functions $g_{1}(c,d;z,p)$ and
$B^{\prime}\left(  z,1-p:\chi\right)  $ are analytic on $\mathbb{H},$
(\ref{53a}) is valid for all $z\in\mathbb{H}$ by analytic continuation.

The proof for $b\equiv c\equiv0(mod$ $k)$ is analogous.
\end{proof}

By taking $Vz=T\left(  z+k\right)  =\left(  az+b+ak\right)  /\left(
cz+d+ck\right)  $ with $a$ and $b$ are odd, instead of $Tz=\left(
az+b\right)  /\left(  cz+d\right)  $ in Corollary \ref{sk2} we obtain the
following result which involves a different sum.

\begin{corollary}
\label{sk1}Let $p$ be odd and let $Vz=\left(  az+b+ak\right)  /\left(
cz+d+ck\right)  $ with $a$ and $b$ are odd. If $a\equiv d\equiv0(mod$ $k),$
then for $z\in\mathbb{H}$
\begin{align}
&  \left(  cz+d+ck\right)  ^{p-1}G\left(  \overline{\chi}\right)  B^{\prime
}\left(  Vz,1-p:\chi\right) \nonumber\\
&  =\overline{\chi}\left(  \frac{b+ak}{2}\right)  \chi(2c)\left\{
G(\chi)B^{\prime}\left(  z,1-p:\overline{\chi}\right)  +\frac{\left(  2\pi
i\right)  ^{p}\chi(-1)}{\left(  p+1\right)  !}g_{1}(c,d+ck;z,p;\overline{\chi
})\right\}  .\text{ } \label{53}%
\end{align}
If $b\equiv c\equiv0(mod$ $k)$, then for $z\in\mathbb{H}$
\begin{align}
&  \left(  cz+d+ck\right)  ^{p-1}G\left(  \overline{\chi}\right)  B^{\prime
}\left(  Vz,1-p:\chi\right) \nonumber\\
&  =\overline{\chi}(a)\chi(d)G(\overline{\chi})B^{\prime}\left(
z,1-p:\chi\right)  +\overline{\chi}(a)\chi(d)\frac{\left(  2\pi i\right)
^{p}\overline{\chi}(-1)}{\left(  p+1\right)  !}g_{1}(c,d+ck;z,p;\chi),
\label{62}%
\end{align}
where
\begin{align}
&  g_{1}(c,d+ck;z,p;\chi)\nonumber\\
&  =-\sum\limits_{m=1}^{p}\binom{p+1}{m}\left(  -\left(  cz+d+ck\right)
\right)  ^{m-1}k^{m-p}\frac{m}{2^{m}}\sum\limits_{n=1}^{ck}\chi(n)\overline
{B}_{p+1-m,\chi}\left(  \frac{\left(  d+ck\right)  n}{2c}\right)  \overline
{E}_{m-1}\left(  \frac{n}{ck}\right)  . \label{54}%
\end{align}

\end{corollary}

We note that this result shows that the sum $S(d,c)$\ can arise in the formula
for $\log\theta_{4}\left(  \frac{az+b+a}{cz+d+c}\right)  ,$\ where $a$\ and
$b$\ are odd.

We conclude this section by assuming that $a$ is even instead of $b$ in
Theorem \ref{tk2}.

\begin{theorem}
\label{tk3}Let $Tz=\left(  az+b\right)  /\left(  cz+d\right)  $ with $a$ is
even. If $a\equiv d\equiv0(mod$ $k),$ then for $z\in\mathbb{K}$ and
$s\in\mathbb{C}$
\begin{align}
&  \hspace{-0.3in}\frac{2^{1-s}}{\left(  cz+d\right)  ^{s}}G\left(
\overline{\chi}\right)  B^{\prime}\left(  Tz,s:\chi\right) \nonumber\\
&  \hspace{-0.2in}=\overline{\chi}(b)\chi(c)G(\chi)\left\{  2H\left(
2z,s:\overline{\chi}\right)  -\chi(2)H\left(  z,s:\overline{\chi}\right)
\right\} \nonumber\\
&  \hspace{-0.2in}+\overline{\chi}(b)\chi(c)\left(  -\frac{k}{2\pi i}\right)
^{s}e^{-\pi is}\sum\limits_{j=1}^{c}\sum\limits_{\mu=0}^{k-1}\sum
\limits_{\nu=0}^{k-1}\overline{\chi}\left(  c\mu+j\right) \nonumber\\
&  \hspace{-0.2in}\times\left\{  2\chi\left(  \left[  \tfrac{2dj}{c}\right]
-\nu\right)  f(2z,s:c,2d)-\chi(2)\chi\left(  \left[  \tfrac{dj}{c}\right]
-\nu\right)  f(z,s:c,d)\right\}  .\text{ } \label{55}%
\end{align}
If $b\equiv c\equiv0(mod$ $k)$, then for $z\in\mathbb{K}$ and $s\in\mathbb{C}$%
\begin{align}
&  2^{1-s}\left(  cz+d\right)  ^{-s}G\left(  \overline{\chi}\right)
B^{\prime}\left(  Tz,s:\chi\right) \nonumber\\
&  =\overline{\chi}\left(  \frac{a}{2}\right)  \chi(2d)G(\overline{\chi
})\left\{  2H\left(  2z,s:\chi\right)  -\overline{\chi}(2)H\left(
z,s:\chi\right)  \right\} \nonumber\\
&  \quad+\overline{\chi}\left(  \frac{a}{2}\right)  \chi(2d)\left(  -\frac
{k}{2\pi i}\right)  ^{s}e^{-\pi is}\sum\limits_{j=1}^{c}\sum\limits_{\mu
=0}^{k-1}\sum\limits_{\nu=0}^{k-1}\chi\left(  j\right) \nonumber\\
&  \qquad\times\left\{  2\overline{\chi}\left(  \left[  \tfrac{2dj}{c}\right]
+2d\mu-\nu\right)  f(2z,s:c,2d)-\overline{\chi}(2)\overline{\chi}\left(
\left[  \tfrac{dj}{c}\right]  +d\mu-\nu\right)  f(z,s:c,d)\right\}  .
\label{56}%
\end{align}

\end{theorem}

\begin{proof}
Let $a$ be even and $Wz=\left(  \frac{a}{2}z+b\right)  /\left(  cz+2d\right)
.$ Since $W\left(  2z\right)  =\frac{1}{2}T(z)$ we have by (\ref{4})%
\[
\left(  cz+d\right)  ^{-s}B^{\prime}(Tz,s:\chi)=2^{s}\left(  2cz+2d\right)
^{-s}H(W\left(  2z\right)  ,s:\chi)-\chi\left(  2\right)  2^{s-1}\left(
cz+d\right)  ^{-s}H(Tz,s:\chi).
\]
Again applying Theorem \ref{tk1} to the right-hand side we get the desired results.
\end{proof}

\begin{corollary}
\label{sk3}Let $p$ be odd and $a$ be even. If $a\equiv d\equiv0(mod$ $k),$
then for $z\in\mathbb{H}$
\begin{align}
&  2^{p}\left(  cz+d\right)  ^{p-1}G\left(  \overline{\chi}\right)  B^{\prime
}\left(  Tz,1-p:\chi\right) \nonumber\\
&  =\overline{\chi}(b)\chi(c)G(\chi)\left\{  2H\left(  2z,1-p:\overline{\chi
}\right)  -\chi(2)H\left(  z,1-p:\overline{\chi}\right)  \right\}
+\overline{\chi}(b)\chi(c)\left(  2\pi i\right)  ^{p}\frac{\chi(-2)}{\left(
p+1\right)  !}g_{2}(c,d;z,p;\overline{\chi}). \label{57}%
\end{align}
If $b\equiv c\equiv0(mod$ $k),$ then for $z\in\mathbb{H}$
\begin{align}
&  2^{p}\left(  cz+d\right)  ^{p-1}G\left(  \overline{\chi}\right)  B^{\prime
}\left(  Tz,1-p:\chi\right) \nonumber\\
&  =\overline{\chi}\left(  \frac{a}{2}\right)  \chi(2d)G(\overline{\chi
})\left\{  2H\left(  2z,1-p:\chi\right)  -\overline{\chi}(2)H\left(
z,1-p:\chi\right)  \right\}  +\overline{\chi}\left(  \frac{a}{2}\right)
\chi(2d)\left(  2\pi i\right)  ^{p}\frac{\overline{\chi}(-2)}{\left(
p+1\right)  !}g_{2}(c,d;z,p;\chi), \label{58}%
\end{align}
where%
\begin{equation}
g_{2}(c,d;z,p;\chi)=-\sum\limits_{m=1}^{p}\binom{p+1}{m}\left(  -\left(
cz+d\right)  \right)  ^{m-1}k^{m-p}\frac{m}{2}\sum\limits_{n=1}^{ck}\left(
-1\right)  ^{n}\chi(n)\overline{B}_{p+1-m,\chi}\left(  \frac{dn}{c}\right)
\overline{E}_{m-1}\left(  \frac{n}{ck}\right)  . \label{59}%
\end{equation}

\end{corollary}

\begin{proof}
Let $a$ be even and $a\equiv d\equiv0(mod$ $k).$ Since
\[
f(z,1-p:c,d)=\frac{2\pi ik^{p-1}}{\left(  p+1\right)  !}\sum_{m=0}^{p+1}%
\binom{p+1}{m}\left(  -\left(  cz+d\right)  \right)  ^{m-1}B_{p+1-m}\left(
\tfrac{\nu+\left\{  dj/c\right\}  }{k}\right)  B_{m}\left(  \tfrac{\mu
c+j}{ck}\right)  ,
\]
(\ref{55}) becomes%
\begin{align}
&  \hspace{-0.55in}2^{p}\left(  cz+d\right)  ^{p-1}G\left(  \overline{\chi
}\right)  B^{\prime}\left(  Tz,1-p:\chi\right) \nonumber\\
&  \hspace{-0.55in}=\overline{\chi}(b)\chi(c)G(\chi)\left\{  2H\left(
2z,1-p:\overline{\chi}\right)  -\chi(2)H\left(  z,1-p:\overline{\chi}\right)
\right\} \nonumber\\
&  \hspace{-0.55in}+\overline{\chi}(b)\chi(c)\frac{\left(  2\pi i\right)
^{p}}{\left(  p+1\right)  !}\sum_{m=0}^{p+1}\binom{p+1}{m}\left(  -\left(
cz+d\right)  \right)  ^{m-1}\nonumber\\
&  \hspace{-0.55in}\times\left\{  2^{m}\sum\limits_{j=1}^{c}\sum
\limits_{\mu=0}^{k-1}\sum\limits_{\nu=0}^{k-1}\overline{\chi}\left(
c\mu+j\right)  \chi\left(  \left[  \tfrac{2dj}{c}\right]  -\nu\right)
B_{p+1-m}\left(  \tfrac{\nu+\left\{  2dj/c\right\}  }{k}\right)  B_{m}\left(
\tfrac{\mu c+j}{ck}\right)  \right. \label{70}\\
&  \hspace{-0.55in}\left.  -\chi(2)\sum\limits_{j=1}^{c}\sum\limits_{\mu
=0}^{k-1}\sum\limits_{\nu=0}^{k-1}\overline{\chi}\left(  c\mu+j\right)
\chi\left(  \left[  \tfrac{dj}{c}\right]  -\nu\right)  B_{p+1-m}\left(
\tfrac{\nu+\left\{  dj/c\right\}  }{k}\right)  B_{m}\left(  \tfrac{\mu
c+j}{ck}\right)  \right\}  . \label{71}%
\end{align}
Note that the above sums over $\mu$ and $\nu$\ are zero for $m=0$ and $m=p+1,$
respectively. Similar to the evaluation of (\ref{64}), (\ref{70}) can be
evaluated as%
\begin{equation}
2^{m}\chi(-1)k^{m-p}\sum\limits_{n=1}^{ck}\overline{\chi}\left(  n\right)
\overline{B}_{p+1-m,\overline{\chi}}\left(  \frac{2dn}{c}\right)  \overline
{B}_{m}\left(  \frac{n}{ck}\right)  \label{72}%
\end{equation}
and (\ref{71}) can be evaluated as
\begin{align}
&  \chi(2)\chi(-1)k^{m-p}\sum\limits_{n=1}^{ck}\overline{\chi}\left(
n\right)  \overline{B}_{p+1-m,\overline{\chi}}\left(  \frac{dn}{c}\right)
\overline{B}_{m}\left(  \frac{n}{ck}\right) \label{80}\\
&  \ =\chi(2)\chi(-1)k^{m-p}\sum\limits_{n=1}^{ck}\overline{\chi}\left(
2n\right)  \overline{B}_{p+1-m,\overline{\chi}}\left(  \frac{2dn}{c}\right)
\overline{B}_{m}\left(  \frac{2n}{ck}\right)  \label{73}%
\end{align}
since $2n$\ runs through a complete residue system $($mod\ $ck)$\ as $n$ does,
for $\left(  2,ck\right)  =1$. Substituting (\ref{72}), (\ref{73}) in
(\ref{70}), (\ref{71}), respectively, and using (\ref{1}), (\ref{11}) we have
\begin{align*}
&  \chi(-1)k^{m-p}2^{m-1}\sum\limits_{n=1}^{ck}\overline{\chi}\left(
n\right)  \overline{B}_{p+1-m,\overline{\chi}}\left(  \frac{2dn}{c}\right)
\left\{  \overline{B}_{m}\left(  \frac{n}{ck}\right)  -\overline{B}_{m}\left(
\frac{n}{ck}+\frac{1}{2}\right)  \right\} \\
&  \ =-\chi(-2)\frac{m}{2}k^{m-p}\left(  \sum\limits_{n=1}^{\frac{ck-1}{2}%
}+\sum\limits_{n=\frac{ck-1}{2}+1}^{ck}\right)  \overline{\chi}\left(
2n\right)  \overline{B}_{p+1-m,\overline{\chi}}\left(  \frac{2nd}{c}\right)
\overline{E}_{m}\left(  \frac{2n}{ck}\right) \\
&  \ =-\chi(-2)\frac{m}{2}k^{m-p}\left\{  \sum\limits_{n=1}^{\frac{ck-1}{2}%
}\overline{\chi}\left(  2n\right)  \overline{B}_{p+1-m,\overline{\chi}}\left(
\frac{2nd}{c}\right)  \overline{E}_{m}\left(  \frac{2n}{ck}\right)  \right. \\
&  \ \qquad\left.  -\sum\limits_{n=0}^{\frac{ck-1}{2}}\overline{\chi}\left(
2n+1\right)  \overline{B}_{p+1-m,\overline{\chi}}\left(  \frac{d\left(
2n+1\right)  }{c}\right)  \overline{E}_{m}\left(  \frac{2n+1}{ck}\right)
\right\} \\
&  \ =-\chi(-2)\frac{m}{2}k^{m-p}\sum\limits_{n=1}^{ck}\left(  -1\right)
^{n}\overline{\chi}\left(  n\right)  \overline{B}_{p+1-m,\overline{\chi}%
}\left(  \frac{dn}{c}\right)  \overline{E}_{m-1}\left(  \frac{n}{ck}\right)  .
\end{align*}
This gives (\ref{57}) for $z\in\mathbb{K}$. Then, by analytic continuation,
(\ref{53}) is valid for all $z\in\mathbb{H}$.
\end{proof}

\section{Reciprocity Theorems}

In this section we prove some reciprocity theorems. The next result can be
viewed as the reciprocity formula for the function $g_{1}(d,c+dk;z,p;\chi)$
given by (\ref{54}).

\begin{theorem}
\label{tk4}Let $p\ $and$\ \left(  d+c\right)  $ be odd with $c,d>0$ and
$(c,d)=1.$ If $d\equiv0(mod$ $k),$ then%
\begin{equation}
g_{1}(d,-c-dk;z,p;\chi)-\chi(-4)\left(  z-k\right)  ^{p-1}g_{1}(c,d+ck;V_{1}%
\left(  z\right)  ,p;\overline{\chi})=g_{1}(1,-k;z,p;\chi), \label{76}%
\end{equation}
if $c\equiv0(mod$ $k),$ then%
\begin{equation}
g_{1}(d,-c-dk;z,p;\overline{\chi})-\overline{\chi}(-4)\left(  z-k\right)
^{p-1}g_{1}(c,d+ck;V_{1}\left(  z\right)  ,p;\chi)=g_{1}(1,-k;z,p;\overline
{\chi}), \label{77}%
\end{equation}
where $V_{1}(z)=\left(  -kz+k^{2}-1\right)  /\left(  z-k\right)  .$
\end{theorem}

\begin{proof}
Let $d+c$\ be odd and $V\left(  z\right)  =\left(  az+b+ak\right)  /\left(
cz+d+ck\right)  ,$\ where $a$ and $b$ are odd. Further let\ $V^{\ast}\left(
z\right)  =\left(  bz-a-bk\right)  /\left(  dz-c-dk\right)  $ and
$V_{1}\left(  z\right)  =\left(  -kz+k^{2}-1\right)  /\left(  z-k\right)  .$

Suppose $a\equiv d\equiv0(mod$ $k).$ By replacing $z$ by $V_{1}(z)$ in
(\ref{53}) we get
\begin{align}
&  \left(  \frac{dz-c-dk}{z-k}\right)  ^{p-1}G\left(  \overline{\chi}\right)
B^{\prime}\left(  V^{\ast}(z),1-p:\chi\right) \nonumber\\
&  =\overline{\chi}(\frac{b+ak}{2})\chi(2c)G(\chi)B^{\prime}\left(
V_{1}(z),1-p:\overline{\chi}\right)  +\overline{\chi}(\frac{b+ak}{2}%
)\chi(2c)G(\chi)\frac{\left(  2\pi i\right)  ^{p}\chi(-1)}{\left(  p+1\right)
!}g_{1}(c,d+ck;V_{1}(z),p;\overline{\chi}). \label{5a}%
\end{align}
By applying $V^{\ast}(z)$ to (\ref{62}) we deduce that
\begin{align}
&  \left(  dz-c-dk\right)  ^{p-1}G\left(  \overline{\chi}\right)  B^{\prime
}\left(  V^{\ast}(z),1-p:\chi\right) \nonumber\\
&  \ =\overline{\chi}(b)\chi(-c)G\left(  \overline{\chi}\right)  B^{\prime
}\left(  z,1-p:\chi\right)  +\overline{\chi}(b)\chi(-c)\frac{\left(  2\pi
i\right)  ^{p}\overline{\chi}(-1)}{\left(  p+1\right)  !}g_{1}%
(d,-c-dk;z,p;\chi). \label{5b}%
\end{align}
To determine $B^{\prime}\left(  V_{1}(z),1-p:\overline{\chi}\right)  $ we
replace $V(z)$ by $V_{1}(z)$ and $\chi$ by $\overline{\chi}$ in (\ref{53}) to
obtain
\begin{align}
&  \left(  z-k\right)  ^{p-1}G\left(  \chi\right)  B^{\prime}\left(
V_{1}(z),1-p:\overline{\chi}\right) \nonumber\\
&  \ =\overline{\chi}(\frac{k^{2}-1}{2})\chi(2)G(\overline{\chi})B^{\prime
}\left(  z,1-p:\chi\right)  +\overline{\chi}(\frac{k^{2}-1}{2})\chi
(2)\frac{\left(  2\pi i\right)  ^{p}\overline{\chi}(-1)}{\left(  p+1\right)
!}g_{1}(1,-k;z,p;\chi). \label{5c}%
\end{align}
Combining (\ref{5a}), (\ref{5b}) and (\ref{5c}) we see that
\[
\overline{\chi}(b)\chi(c)\left(  g_{1}(d,-c-dk;z,p;\chi)-\chi(-4)\left(
z-k\right)  ^{p-1}g_{1}(c,d+ck;V_{1}(z),p;\overline{\chi})\right)
=\overline{\chi}(b)\chi(c)g_{1}(1,-k;z,p;\chi).
\]
This gives (\ref{76}) since $\overline{\chi}(b)\neq0$ and $\chi(c)\neq0$ for
$\left(  b,k\right)  =\left(  c,k\right)  =1$.

To prove (\ref{77}), first replace $z$\ by $V_{1}(z)$\ in (\ref{62}) and apply
$V^{\ast}(z)$\ to (\ref{53}), and then replace $V(z)$\ by $V_{1}(z)$\ in
(\ref{53}).
\end{proof}

The theorem above may be simplified for the special value of $z$.\textbf{\ }In
particular, we consider $z=\left(  c+dk\right)  /d$\ and let $d\equiv
0(mod$\textbf{\ }$k).$ We first calculate $g_{1}(1,-k;z,p;\chi).$\ For this,
we use (\ref{66a}) and (\ref{67}) instead of (\ref{54}) and find that%
\begin{align*}
g_{1}(1,-k;z,p;\chi)  &  =\sum_{m=1}^{p}\binom{p+1}{m}\left(  k-z\right)
^{m-1}k^{m-p}\\
&  \times\left(  \sum\limits_{n=1}^{2k}\chi\left(  n\right)  \overline
{B}_{p+1-m,\chi}\left(  \frac{-kn}{2}\right)  \overline{B}_{m}\left(  \frac
{n}{2k}\right)  -\chi(2)2^{-p}k^{1-m}B_{p+1-m,\chi}B_{m,\overline{\chi}%
}\right)  .
\end{align*}
The sum over $n$ may be written as
\begin{align*}
&  \left(  \sum\limits_{n\text{ odd}}+\sum\limits_{n\text{ even}}\right)
\chi\left(  n\right)  \overline{B}_{p+1-m,\chi}\left(  \frac{-kn}{2}\right)
\overline{B}_{m}\left(  \frac{n}{2k}\right) \\
&  \ =\overline{B}_{p+1-m,\chi}\left(  \frac{-k}{2}\right)  \left(
\sum\limits_{n=1}^{2k}-\sum\limits_{n\text{ even}}\right)  +\sum
\limits_{n\text{ even}},
\end{align*}
where
\begin{align*}
\overline{B}_{p+1-m,\chi}\left(  \frac{-k}{2}\right)   &  =\overline
{B}_{p+1-m,\chi}\left(  \frac{k}{2}\right)  =\left(  \chi\left(  2\right)
2^{m-p}-1\right)  B_{p+1-m,\chi},\\
\sum\limits_{n\text{ even}}  &  =\sum\limits_{n=1}^{k}\chi\left(  2n\right)
\overline{B}_{m}\left(  \frac{2n}{2k}\right)  =\chi\left(  2\right)
k^{1-m}B_{m,\overline{\chi}}%
\end{align*}
and
\begin{align*}
\sum\limits_{n=1}^{2k}\chi\left(  n\right)  \overline{B}_{m}\left(  \frac
{n}{2k}\right)   &  =\sum\limits_{n=1}^{k}\chi\left(  n\right)  \left\{
\overline{B}_{m}\left(  \frac{n}{2k}\right)  +\overline{B}_{m}\left(  \frac
{n}{2k}+\frac{1}{2}\right)  \right\} \\
&  =2^{1-m}k^{1-m}B_{m,\overline{\chi}}%
\end{align*}
by (\ref{68}), (\ref{3}) and (\ref{1}). Thus,
\begin{align}
&  g_{1}(1,-k;\frac{c+dk}{d},p;\chi)\nonumber\\
&  =-k^{1-p}\sum_{m=1}^{p}\binom{p+1}{m}\left(  -\frac{c}{d}\right)
^{m-1}\left(  \chi(2)2^{m}-1\right)  \left(  2^{-p}\chi(2)-2^{1-m}\right)
B_{p+1-m,\chi}B_{m,\overline{\chi}}. \label{14}%
\end{align}

The functions $g_{1}(c,d+ck;\frac{-kz+k^{2}-1}{z-k},p;\overline{\chi})$ and
$g_{1}(d,-c-dk;z,p;\chi)$ become
\begin{equation}
g_{1}(c,d+ck;\frac{-d-ck}{c},p;\overline{\chi})=-\frac{p+1}{2k^{p-1}}%
\sum\limits_{n=1}^{ck}\overline{\chi}(n)\overline{B}_{p,\overline{\chi}%
}\left(  \frac{d+ck}{2c}n\right)  \text{ } \label{12}%
\end{equation}
and%
\begin{align}
g_{1}(d,-c-dk;\dfrac{c+dk}{d},p;\chi)  &  =-\frac{p+1}{2k^{p-1}}%
\sum\limits_{n=1}^{dk}\chi(n)\overline{B}_{p,\chi}\left(  \frac{-c-dk}%
{2d}n\right) \nonumber\\
&  =\chi(-1)\frac{p+1}{2k^{p-1}}\sum\limits_{n=1}^{dk}\chi(n)\overline
{B}_{p,\chi}\left(  \frac{c+dk}{2d}n\right)  \label{13}%
\end{align}
by the fact $\overline{B}_{p,\chi}\left(  -x\right)  =\left(  -1\right)
^{p}\overline{\chi}\left(  -1\right)  \overline{B}_{p,\chi}\left(  x\right)
.$

\begin{definition}
The Hardy-Berndt character sum $S_{p}\left(  d,c:\chi\right)  $ is defined for
$c>0$\ by%
\[
S_{p}\left(  d,c:\chi\right)  =\sum\limits_{n=1}^{ck}\chi\left(  n\right)
\overline{B}_{p,\chi}\left(  \frac{d+ck}{2c}n\right)  .
\]

\end{definition}

Using (\ref{14}), (\ref{12}) and (\ref{13}) in (\ref{76}) we have proved the
following reciprocity formula for $d\equiv0(mod$ $k)$. The proof for
$c\equiv0(mod$ $k)$ follows from (\ref{77}).

\begin{corollary}
\label{ck1}Let $p\ $be odd and let $d$ and $c$ be coprime positive integers
with $\left(  d+c\right)  $ is odd. If either $c$ or $d\equiv0(mod$ $k),$ then%
\begin{align*}
&  \chi(-1)\left(  p+1\right)  \left\{  cd^{p}S_{p}(c,d:\chi)+dc^{p}%
\chi(4)S_{p}(d,c:\overline{\chi})\right\} \\
&  \ =\sum\limits_{m=1}^{p}\binom{p+1}{m}\left(  -1\right)  ^{m}\left\{
\chi(2)2^{m}-1\right\}  \left\{  \chi(2)2^{1-p}-2^{2-m}\right\}
c^{m}d^{p+1-m}B_{p+1-m,\chi}B_{m,\overline{\chi}}.
\end{align*}

\end{corollary}

Corollary \ref{sk2} do not give reciprocity formula for the function
$g_{1}(d,c;z,p;\chi)$\ in the sense of Theorem \ref{tk4}\ because of the
restriction on $b$\ ($b$ is even). By the similar restriction on $a$\ in
Corollary \ref{sk3} we do not have reciprocity formula for the function
$g_{2}(d,c;z,p;\chi)$. However, we have the following result.

\begin{theorem}
\label{tk5}Let $p\ $and$\ d$ be odd with $c,d>0$ and $(c,d)=1.$ If
$d\equiv0(mod$ $k),$ then
\[
\overline{\chi}(-2)g_{2}(d,-c;z,p;\chi)-2^{p}z^{p-1}g_{1}(c,d;\frac{-1}%
{z},p;\overline{\chi})=\overline{\chi}(-2)g_{2}(1,0;z,p;\chi),
\]
if $c\equiv0(mod$ $k),$ then%
\[
\chi(-2)g_{2}(d,-c;z,p;\overline{\chi})-2^{p}z^{p-1}g_{1}(c,d;\frac{-1}%
{z},p;\chi)=\chi(-2)g_{2}(1,0;z,p;\overline{\chi}).
\]

\end{theorem}

\begin{proof}
Let $d$\ be odd and $T\left(  z\right)  =\left(  az+b\right)  /\left(
cz+d\right)  ,$\ where $b$ is even. Let $T^{\ast}\left(  z\right)  =\left(
bz-a\right)  /\left(  dz-c\right)  $ and $T_{1}\left(  z\right)  =-1/z$.

Suppose $a\equiv d\equiv0(mod$ $k).$ By replacing $z$ by $T_{1}(z)$ in
(\ref{53a}) we get
\begin{align}
&  \hspace{-0.52in}\left(  \frac{dz-c}{z}\right)  ^{p-1}G\left(
\overline{\chi}\right)  B^{\prime}\left(  T^{\ast}(z),1-p:\chi\right)
\nonumber\\
&  \hspace{-0.5in}=\overline{\chi}(\frac{b}{2})\chi(2c)\left\{  G(\chi
)B^{\prime}\left(  T_{1}(z),1-p:\overline{\chi}\right)  +\tfrac{\left(  2\pi
i\right)  ^{p}\chi(-1)}{\left(  p+1\right)  !}g_{1}(c,d;T_{1}(z),p;\overline
{\chi})\right\}  . \label{78}%
\end{align}
Applying $T^{\ast}(z)$ to (\ref{58}) gives%
\begin{align}
&  2^{p}\left(  dz-c\right)  ^{p-1}G\left(  \overline{\chi}\right)  B^{\prime
}\left(  T^{\ast}(z),1-p:\chi\right) \nonumber\\
&  \ =\overline{\chi}(\frac{b}{2})\chi(-2c)G(\overline{\chi})\left\{
2H\left(  2z,1-p:\chi\right)  -%
\genfrac{}{}{0pt}{}{{}}{{}}%
\overline{\chi}(2)H\left(  z,1-p:\chi\right)  \right\} \nonumber\\
&  \quad+\overline{\chi}(\frac{b}{2})\chi(-2c)\left(  2\pi i\right)  ^{p}%
\frac{\overline{\chi}(-2)}{\left(  p+1\right)  !}g_{2}(d,-c;z,p;\chi).
\label{79}%
\end{align}
\ If we replace $T(z)$ by $T_{1}(z)$ and $\chi$ by $\overline{\chi}$ in
(\ref{57}) we see that
\begin{align}
&  2^{p}z^{p-1}G\left(  \chi\right)  B^{\prime}\left(  T_{1}(z),1-p:\overline
{\chi}\right) \nonumber\\
&  =\chi(-1)G(\overline{\chi})\left\{  2H\left(  2z,1-p:\chi\right)  -%
\genfrac{}{}{0pt}{}{{}}{{}}%
\overline{\chi}(2)H\left(  z,1-p:\chi\right)  \right\}  +\chi(-1)\left(  2\pi
i\right)  ^{p}\frac{\overline{\chi}(-2)}{\left(  p+1\right)  !}g_{2}%
(1,0;z,p;\chi). \label{81}%
\end{align}
Combining (\ref{78}), (\ref{79}) and (\ref{81}) deduces that%
\[
\overline{\chi}(\frac{b}{2})\chi(-2c)\left\{  \overline{\chi}(-2)g_{2}%
(d,-c;z,p;\chi)-2^{p}z^{p-1}g_{1}(c,d;T_{1}(z),p;\overline{\chi})\right\}
=\overline{\chi}(\frac{b}{2})\chi(-2c)\overline{\chi}(-2)g_{2}(1,0;z,p;\chi).
\]
This completes the proof for $a\equiv d\equiv0(mod$ $k)$ since $\overline
{\chi}(b/2)\neq0$ and $\chi(-2c)\neq0$ for $\left(  b,k\right)  =\left(
2c,k\right)  =1.$

The proof for $b\equiv c\equiv0(mod$ $k)$ is similar.
\end{proof}

This theorem is simplified by setting\ $z=c/d.$ To calculate $g_{2}%
(1,0;c/d,p;\chi),$\ we use (\ref{72}) and (\ref{80}) instead of (\ref{59}),
and find that
\begin{align}
&  \overline{\chi}(-2)g_{2}\left(  1,0;\frac{c}{d},p;\chi\right) \nonumber\\
&  =\overline{\chi}(-1)\sum_{m=1}^{p}\binom{p+1}{m}\left(  -\frac{c}%
{d}\right)  ^{m-1}\left\{  2^{m}-\overline{\chi}(2)\right\}  k^{m-p}%
\sum\limits_{n=1}^{k}\chi\left(  n\right)  \overline{B}_{p+1-m,\chi}\left(
0\right)  \overline{B}_{m}\left(  \frac{n}{k}\right) \nonumber\\
&  =\overline{\chi}(-1)k^{1-p}\sum_{m=1}^{p}\binom{p+1}{m}\left(  -\frac{c}%
{d}\right)  ^{m-1}\left\{  2^{m}-\overline{\chi}(2)\right\}  B_{p+1-m,\chi
}B_{m,\overline{\chi}}. \label{17}%
\end{align}

We also have
\begin{align}
g_{1}(c,d;-d/c,p;\overline{\chi})  &  =-\frac{p+1}{2k^{p-1}}\sum
\limits_{n=1}^{ck}\overline{\chi}(n)\overline{B}_{p,\overline{\chi}}\left(
\frac{dn}{2c}\right)  ,\label{15}\\
g_{2}(d,-c;c/d,p;\chi)  &  =-\frac{p+1}{2k^{p-1}}\sum\limits_{n=1}^{dk}\left(
-1\right)  ^{n}\chi(n)\overline{B}_{p,\chi}\left(  \frac{-cn}{d}\right)
\nonumber\\
&  =\chi(-1)\frac{p+1}{2k^{p-1}}\sum\limits_{n=1}^{dk}\left(  -1\right)
^{n}\chi(n)\overline{B}_{p,\chi}\left(  \frac{cn}{d}\right)  . \label{16}%
\end{align}

\begin{definition}
The Hardy-Berndt character sums $s_{3,p}\left(  d,c:\chi\right)  $ and
$s_{4,p}\left(  d,c:\chi\right)  $ are defined for $c>0$\ by%
\begin{align*}
&  s_{3,p}\left(  d,c:\chi\right)  =\sum\limits_{n=1}^{ck}\left(  -1\right)
^{n}\chi\left(  n\right)  \overline{B}_{p,\chi}\left(  \frac{dn}{c}\right)
,\\
&  s_{4,p}\left(  d,c:\chi\right)  =\sum\limits_{n=1}^{ck}\chi\left(
n\right)  \overline{B}_{p,\chi}\left(  \frac{dn}{2c}\right)  .
\end{align*}

\end{definition}

Using (\ref{15}), (\ref{16}) and (\ref{17}) in Theorem \ref{tk5} we have
proved the following reciprocity formula for $s_{3,p}(d,c:\chi)$\ and
$s_{4,p}(d,c:\chi).$

\begin{corollary}
\label{ck2}Let $p\ $be odd and let $d$ and $c$ be coprime positive integers
with $d$ is odd. If either $c$ or $d\equiv0(mod$ $k),$ then
\begin{align*}
&  \chi(-1)\left(  p+1\right)  \left\{  d\left(  2c\right)  ^{p}%
s_{4,p}(d,c:\overline{\chi})+cd^{p}\overline{\chi}(2)s_{3,p}(c,d:\chi)\right\}
\\
&  \ =2\sum\limits_{m=1}^{p}\binom{p+1}{m}\left(  -1\right)  ^{m}\left\{
\overline{\chi}(2)-2^{m}\right\}  c^{m}d^{p+1-m}B_{p+1-m,\chi}B_{m,\overline
{\chi}}.
\end{align*}

\end{corollary}

\section{Further Results}

For $c>0$\ and $1\leq m\leq p$\ we define%
\begin{align*}
&  S_{p+1-m,m}\left(  d,c:\chi\right)  =-\frac{m}{2^{m}}\sum\limits_{n=1}%
^{ck}\chi\left(  n\right)  \overline{B}_{p+1-m,\chi}\left(  \frac{\left(
d+ck\right)  n}{2c}\right)  \overline{E}_{m-1}\left(  \frac{n}{ck}\right)  ,\\
&  s_{4,p+1-m,m}\left(  d,c:\chi\right)  =-\frac{m}{2^{m}}\sum\limits_{n=1}%
^{ck}\chi\left(  n\right)  \overline{B}_{p+1-m,\chi}\left(  \frac{dn}%
{2c}\right)  \overline{E}_{m-1}\left(  \frac{n}{ck}\right)  .
\end{align*}
Note that we have
\[
S_{p}(d,c:\chi)=-2S_{p,1}\left(  d,c:\chi\right)  \text{\ and }s_{4,p}\left(
d,c:\chi\right)  =-2s_{4,p,1}\left(  d,c:\chi\right)  .
\]

The fixed points of a modular transformation $Tz=\left(  az+b\right)  /\left(
cz+d\right)  $ are $z_{1,2}=\left(  a-d\pm\sqrt{\left(  a+d\right)  ^{2}%
-4}\right)  /2c.$ Then $z=\left(  a-d+\sqrt{\left(  a+d\right)  ^{2}%
-4}\right)  /2c$ is in the upper half-plane$\ \Leftrightarrow\left\vert
a+d\right\vert <2.$

Following Goldberg (\cite[Ch. 7]{g}), we employ fixed points of a modular
transformation to deduce certain properties of $s_{4,p+1-m,m}\left(
d,c:\chi\right)  $\ and $S_{p+1-m,m}\left(  d,c:\chi\right)  $.

\begin{theorem}
\label{tk6}Let$\ d,$ $c$\ and $k$ be odd integers and let $c\equiv0\left(
\text{mod }k\right)  $ and $d^{2}\equiv-1\left(  \text{mod }c\right)  .$ If
$p\equiv\chi\left(  -1\right)  \left(  \text{mod }4\right)  ,$ then
\[
\sum\limits_{m=1}^{p}\binom{p+1}{m}\left(  -i\right)  ^{m}k^{m}s_{4,p+1-m,m}%
\left(  d,c:\chi\right)  =0.
\]
If $p\equiv-\chi\left(  -1\right)  \left(  \text{mod }4\right)  ,$\ then%
\[
4\chi\left(  -1\right)  G\left(  \overline{\chi}\right)  B\left(
z_{0},1-p:\chi\right)  =-\frac{\left(  2\pi i\right)  ^{p}}{\left(
p+1\right)  !}\sum\limits_{m=1}^{p}\binom{p+1}{m}\left(  -i\right)
^{m-1}k^{m-p}s_{4,p+1-m,m}\left(  d,c:\chi\right)  ,
\]
where $z_{0}=\left(  -d+i\right)  /c$.
\end{theorem}

\begin{proof}
Let $a=-d$ and $d^{2}\equiv-1\left(  \text{mod }c\right)  $. Then, there exist
an even integer $b$\ such that $-d^{2}-bc=ad-bc=1.$\ Hence, $Tz=\left(
az+b\right)  /\left(  cz+d\right)  $\ is a modular transformation and
$z_{0}=\left(  -d+i\right)  /c$ is a fixed point of $Tz$\ in the upper
half-plane. Therefore, by setting $a=-d$\ and $z=z_{0}=\left(  -d+i\right)
/c$ in (\ref{62a}) we see that%
\[
\frac{\left(  2\pi i\right)  ^{p}}{\left(  p+1\right)  !}g_{1}(c,d;z_{0}%
,p;\chi)=\left(  \left(  i\right)  ^{p-1}-\chi(-1)\right)  G\left(
\overline{\chi}\right)  B^{\prime}\left(  z_{0},1-p:\chi\right)  .
\]
It follows from this and (\ref{54a}) that if $p\equiv\chi\left(  -1\right)
\left(  \text{mod }4\right)  ,$ then%
\[
\sum\limits_{m=1}^{p}\binom{p+1}{m}\left(  -i\right)  ^{m-1}k^{m-p}%
s_{4,p+1-m,m}\left(  d,c:\chi\right)  =0
\]
and if $p\equiv-\chi\left(  -1\right)  \left(  \text{mod }4\right)  ,$ then%
\[
4\chi\left(  -1\right)  G\left(  \overline{\chi}\right)  B\left(
z_{0},1-p:\chi\right)  =-\frac{\left(  2\pi i\right)  ^{p}}{\left(
p+1\right)  !}\sum\limits_{m=1}^{p}\binom{p+1}{m}\left(  -i\right)
^{m-1}k^{m-p}s_{4,p+1-m,m}\left(  d,c:\chi\right)  .
\]

\end{proof}

From (\ref{53a}), similar result to that of Theorem \ref{tk6} can be obtained
for $d\equiv0\left(  \text{mod }k\right)  $\ with real $\chi$. Moreover, from
Theorem \ref{tk6}, it is seen that
\[
2G\left(  \overline{\chi}\right)  B\left(  z_{0},0:\chi\right)  =\pi
is_{4,1,1}\left(  d,c:\chi\right)
\]
for $p=1=-\chi\left(  -1\right)  ,$ and $s_{4,1,1}\left(  d,c:\chi\right)  =0$
for $p=1=\chi\left(  -1\right)  .$ In addition, if $\chi$\ is real, then
$s_{4,2,2}\left(  d,c:\chi\right)  =0$\ and $s_{4,3,1}\left(  d,c:\chi\right)
=k^{2}s_{4,1,3}\left(  d,c:\chi\right)  $\ for $p=3$\ and $\chi\left(
-1\right)  =-1$.

The next theorem follows from (\ref{62}) by using $V\left(  z\right)  =\left(
az+b+ak\right)  /\left(  cz+d+ck\right)  $ with $a$\ and $b$\ are odd, and
$a=-d-ck$.

\begin{theorem}
\label{tk7}Let$\ \left(  d+c\right)  $ be odd and let $c\equiv0\left(
\text{mod }k\right)  $ and $d^{2}\equiv-1\left(  \text{mod }c\right)  .$ If
$p\equiv\chi\left(  -1\right)  \left(  \text{mod }4\right)  ,$ then
\[
\sum\limits_{m=1}^{p}\binom{p+1}{m}\left(  -i\right)  ^{m}k^{m}S_{p+1-m,m}%
\left(  d,c:\chi\right)  =0.
\]
If $p\equiv-\chi\left(  -1\right)  \left(  \text{mod }4\right)  ,$\ then%
\[
-4\chi\left(  -1\right)  G\left(  \overline{\chi}\right)  B\left(
z_{1},1-p:\chi\right)  =\frac{\left(  2\pi i\right)  ^{p}}{\left(  p+1\right)
!}\sum\limits_{m=1}^{p}\binom{p+1}{m}\left(  -i\right)  ^{m-1}k^{m-p}%
S_{p+1-m,m}\left(  d,c:\chi\right)  ,
\]
where $z_{1}=\left(  -d-ck+i\right)  /c$.
\end{theorem}

In the sequel we will frequently use the following lemma \cite[Lemma 5.5]%
{cck}. We state it here in the following form for $x=y=0$.

\begin{lemma}
\label{lek2}Let $(c,d)=1$ with $c>0.$ Let $\chi$ is a primitive character of
modulus $k,$ where $k$ is a prime number if $\left(  k,cd\right)  =1$ and $p $
is even, otherwise $k$ is an arbitrary integer. Then
\[
\sum\limits_{n=1}^{ck-1}\chi(n)\overline{B}_{p,\chi}\left(  \frac{dn}%
{c}\right)  =c^{1-p}\chi\left(  c\right)  \overline{\chi}\left(  -d\right)
\left(  k^{p}-1\right)  \overline{B}_{p}\left(  0\right)  .
\]

\end{lemma}

We want to show that the reciprocity formulas given by Corollaries \ref{ck1}
and \ref{ck2} are also valid when $(c,d)>1$. For this we need the following.

\begin{proposition}
\label{lek3}Let $p$\ be odd and $c>0$ with $(c,d)=1.$ Then, for positive
integer $q$, we have if $\left(  d+c\right)  $\ is odd,
\[
S_{p}(qd,qc:\chi)=\left\{
\begin{array}
[c]{cl}%
S_{p}(d,c:\chi), & q\text{ odd,}\\
0, & q\text{ even,}%
\end{array}
\right.
\]
if $c$\ is odd,%
\[
s_{3,p}(qd,qc:\chi)=\left\{
\begin{array}
[c]{cl}%
s_{3,p}(d,c:\chi), & q\text{ odd,}\\
0, & q\text{ even,}%
\end{array}
\right.
\]
if $d$\ is odd,%
\[
s_{4,p}(qd,qc:\chi)=\left\{
\begin{array}
[c]{cl}%
s_{4,p}(d,c:\chi), & q\text{ odd,}\\
0, & q\text{ even.}%
\end{array}
\right.
\]

\end{proposition}

\begin{proof}
Led $d$\ be odd.\ By the definition of $s_{4,p}(d,c:\chi),$ we have%
\begin{align*}
&  s_{4,p}(qd,qc:\chi)=\sum\limits_{n=1}^{qck}\chi(n)\overline{B}_{p,\chi
}\left(  \frac{dn}{2c}\right) \\
&  \quad=\sum\limits_{n=1}^{ck}\sum\limits_{m=0}^{q-1}\chi(n)\overline
{B}_{p,\chi}\left(  \frac{dn}{2c}+\frac{kdm}{2}\right)  =\sum\limits_{n=1}%
^{ck-1}\chi(n)\sum\limits_{m=0}^{q-1}\overline{B}_{p,\chi}\left(  \frac
{dn}{2c}+\frac{km}{2}\right)  .
\end{align*}
It follows from (\ref{68}) that
\[
\sum\limits_{m=0}^{q-1}\overline{B}_{p,\chi}\left(  \frac{dn}{2c}+\frac{km}%
{2}\right)  =\left\{
\begin{array}
[c]{cc}%
\frac{q-1}{2}\chi\left(  2\right)  2^{1-p}\overline{B}_{p,\chi}\left(
\frac{dn}{c}\right)  +\overline{B}_{p,\chi}\left(  \frac{dn}{2c}\right)  , &
q\text{ odd.}\\
\frac{q}{2}\chi\left(  2\right)  2^{1-p}\overline{B}_{p,\chi}\left(  \frac
{dn}{c}\right)  , & q\text{ even.}%
\end{array}
\right.
\]
Thus, by Lemma \ref{lek2} we have%
\[
s_{4,p}(qd,qc:\chi)=\left\{
\begin{array}
[c]{cl}%
s_{4,p}(d,c:\chi), & q\text{ odd,}\\
0, & q\text{ even.}%
\end{array}
\right.
\]

Let $c$\ be odd. We have%
\begin{align*}
s_{3,p}(qd,qc:\chi)  &  =\sum\limits_{n=1}^{qck}\left(  -1\right)  ^{n}%
\chi(n)\overline{B}_{p,\chi}\left(  \frac{dn}{c}\right) \\
&  =\sum\limits_{n=1}^{ck}\sum\limits_{m=0}^{q-1}\left(  -1\right)
^{n+ckm}\chi(n)\overline{B}_{p,\chi}\left(  \frac{dn}{c}+km\right)
=s_{3,p}(d,c:\chi)\sum\limits_{m=0}^{q-1}\left(  -1\right)  ^{m},
\end{align*}
which completes the proof.
\end{proof}

Now from Proposition \ref{lek3} and Corollaries \ref{ck1} and \ref{ck2}, we
arrive at the following reciprocity formulas for $\left(  c,d\right)  >1$.

\begin{corollary}
Let $p\ $be odd and let $c,$ $d>0$\ with $\left(  d+c\right)  $\ odd,
$(c,d)=q$\ odd and $\left(  q,k\right)  =1$. If either $c$ or $d\equiv0(mod$
$k),$ then%
\begin{align*}
&  \hspace{-0.27in}\chi(-1)\left(  p+1\right)  \left\{  cd^{p}S_{p}%
(c,d:\chi)+dc^{p}\chi(4)S_{p}(d,c:\overline{\chi})\right\} \\
&  \hspace{-0.27in}=\sum\limits_{m=1}^{p}\binom{p+1}{m}\left(  -1\right)
^{m}\left(  \chi(2)2^{m}-1\right)  \left(  \chi(2)2^{1-p}-2^{2-m}\right)
c^{m}d^{p+1-m}B_{p+1-m,\chi}B_{m,\overline{\chi}}.
\end{align*}

\end{corollary}

\begin{corollary}
Let $p\ $be odd and let $c,$ $d>0$\ with $d$\ odd, $(c,d)=q$\ odd and $\left(
q,k\right)  =1$. If either $c$ or $d\equiv0(mod$ $k),$ then
\begin{align*}
&  \chi(-1)\left(  p+1\right)  \left\{  d\left(  2c\right)  ^{p}%
s_{4,p}(d,c:\overline{\chi})+cd^{p}\overline{\chi}(2)s_{3,p}(c,d:\chi)\right\}
\\
&  \ =2\sum\limits_{m=1}^{p+1}\binom{p+1}{m}\left(  -1\right)  ^{m}\left(
\overline{\chi}(2)-2^{m}\right)  c^{m}d^{p+1-m}B_{p+1-m,\chi}B_{m,\overline
{\chi}}.
\end{align*}

\end{corollary}

So far we assume that $\left(  d+c\right)  $ and $p$, $c$ and $p,$ $d$ and $p$
are odd for the sums $S_{p}(d,c:\chi),$\ $s_{3,p}(d,c:\chi)$\ and
$s_{4,p}(d,c:\chi),$\ respectively. For the remaining three cases these sums
may be evaluated as in the following:

\begin{proposition}
\label{pr1}Let $k$\ be as in Lemma \ref{lek2} and odd. For $c>0$\ with
$(c,d)=1,$\ we have\vspace{0.13in}

(i) $S_{p}(d,c:\chi)=\left\{
\begin{array}
[c]{cl}%
0, & p\text{\ odd and }\left(  d+c\right)  \text{\ even,}\\[0.05in]%
\chi\left(  2\right)  \lambda, & p\text{ and }\left(  d+c\right)  \text{
even,}\\[0.15in]%
2^{-p}\chi\left(  2\right)  \lambda, & p\text{\ even and }\left(  d+c\right)
\ \text{odd,}%
\end{array}
\right.  $\vspace{0.1in}

\noindent where $\lambda=c^{1-p}\chi\left(  c\right)  \overline{\chi}\left(
-d\right)  \left(  k^{p}-1\right)  B_{p}.$\vspace{0.1in}

(ii) $s_{3,p}(d,c:\chi)=\left\{
\begin{array}
[c]{cl}%
0, & c+p\text{ odd,}\\[0.05in]%
\left(  2^{p}-1\right)  \lambda, & p\text{ and }c\text{ even.}%
\end{array}
\right.  $ \vspace{0.1in}

(iii) $s_{4,p}(d,c:\chi)=\left\{
\begin{array}
[c]{cl}%
0, & p\text{\ odd and }d\text{\ even,}\\[0.05in]%
\chi\left(  2\right)  \lambda, & p\text{ and }d\text{ even,}\\[0.15in]%
2^{-p}\chi\left(  2\right)  \lambda, & p\text{\ even and }d\ \text{odd.}%
\end{array}
\right.  $
\end{proposition}

\begin{proof}
Let $\left(  d+c\right)  $ be even. Then, $\left(  d+ck\right)  $ is even and
\[
S_{p}(d,c:\chi)=\sum\limits_{n=1}^{ck}\chi(n)\overline{B}_{p,\chi}\left(
\frac{\left(  d+kc\right)  /2}{c}n\right)  =\left(  -1\right)  ^{p}%
\sum\limits_{n=1}^{ck}\chi(n)\overline{B}_{p,\chi}\left(  \frac{\left(
d+kc\right)  /2}{c}n\right)  .
\]
by the identity $\overline{B}_{p,\chi}\left(  -x\right)  =\left(  -1\right)
^{p}\overline{\chi}\left(  -1\right)  \overline{B}_{p,\chi}\left(  x\right)
.$ Hence, from Lemma \ref{lek2}, we get $S_{p}(d,c:\chi)=0$ for odd $p,$ and
$S_{p}(d,c:\chi)=\chi\left(  2\right)  \lambda$ for even $p.$

Let $p$ be even and $\left(  d+c\right)  $ be\ odd. We have%
\[
S_{p}(d,c:\chi)=\sum\limits_{n=1}^{ck}\chi(n)\overline{B}_{p,\chi}\left(
\tfrac{\left(  d+kc\right)  n}{2c}+\tfrac{\left(  d+kc\right)  k}{2}\right)
.
\]
Upon the use of (\ref{68}) we deduce%
\begin{align*}
2S_{p}(d,c:\chi)  &  =\sum\limits_{n=1}^{ck}\chi(n)\left\{  \overline
{B}_{p,\chi}\left(  \tfrac{\left(  d+kc\right)  n}{2c}\right)  +\overline
{B}_{p,\chi}\left(  \tfrac{\left(  d+kc\right)  n}{2c}+\tfrac{k}{2}\right)
\right\} \\
&  =2^{1-p}\chi\left(  2\right)  \sum\limits_{n=1}^{ck}\chi(n)\overline
{B}_{p,\chi}\left(  \tfrac{dn}{c}\right)  .
\end{align*}
Thus, using Lemma \ref{lek2} we complete the proof for $S_{p}(d,c:\chi)$.
\end{proof}

\section{Integral Representations}

In \cite{b5}, Berndt derived the character analogue of the Euler--Maclaurin
summation formula. We apply this formula to generalized Bernoulli function to
obtain identities involving integrals for $s_{3,p}\left(  d,c:\chi\right)
$\ and $s_{4,p}\left(  d,c:\chi\right)  .$\ These identities lead an
alternative proof of Corollary \ref{ck2}\ which is given in the end of this section.

Berndt's formula is presented here in the following form.

\begin{theorem}
\cite[Theorem 4.1]{b5}\label{E-M} Let $k$ be an arbitrary integer and $f\in
C^{\left(  l+1\right)  }\left[  \alpha,\beta\right]  ,$ $-\infty<\alpha
<\beta<\infty.$ Then
\begin{align*}
\sum_{\alpha\leq n\leq\beta}\ \hspace{-0.13in}^{^{\prime}}\chi\left(
n\right)  f\left(  n\right)   &  =\chi\left(  -1\right)  \sum\limits_{j=0}%
^{l}\frac{\left(  -1\right)  ^{j+1}}{\left(  j+1\right)  !}\left\{
\overline{B}_{j+1,\overline{\chi}}\left(  \beta\right)  f^{\left(  j\right)
}\left(  \beta\right)  -\overline{B}_{j+1,\overline{\chi}}\left(
\alpha\right)  f^{\left(  j\right)  }\left(  \alpha\right)  \right\} \\
&  +\chi\left(  -1\right)  \frac{\left(  -1\right)  ^{l}}{\left(  l+1\right)
!}\int\limits_{\alpha}^{\beta}\overline{B}_{l+1,\overline{\chi}}\left(
u\right)  f^{\left(  l+1\right)  }\left(  u\right)  du,
\end{align*}
where the dash indicates that if $n=\alpha$ or $n=\beta,$ only $\frac{1}%
{2}\chi\left(  \alpha\right)  f\left(  \alpha\right)  $ or $\frac{1}{2}%
\chi\left(  \beta\right)  f\left(  \beta\right)  ,$ respectively, is counted.
\end{theorem}

Let $f\left(  x\right)  =\overline{B}_{p,\chi}\left(  xy\right)  ,$ $y\in R,$
and $c$ be a positive integer. The property
\[
\frac{d}{dx}\overline{B}_{m,\chi}\left(  x\right)  =m\overline{B}_{m-1,\chi
}\left(  x\right)  ,\text{ }m\geq2
\]
(\cite[Corollary 3.3]{b5}) entails that
\[
\frac{d^{j}}{dx^{j}}f\left(  x\right)  =\frac{d^{j}}{dx^{j}}\overline
{B}_{p,\chi}\left(  xy\right)  =y^{j}\frac{p!}{\left(  p-j\right)  !}%
\overline{B}_{p-j,\chi}\left(  xy\right)
\]
for $0\leq j\leq p-1$ and $f\in C^{\left(  p-1\right)  }\left[  \alpha
,\beta\right]  .$

We consider the following three cases;

\qquad I) $y=b/c,$ $\alpha=0$ and $\beta=ck,$

\qquad II) $y=b/\left(  2c\right)  ,$ $\alpha=0$ and $\beta=ck,$

\qquad III) $y=2b/c,$ $\alpha=0$ and $\beta=ck/2,$

\noindent where $c>0,$ separately.

For $\alpha=0,$ $\beta=ck$ and $1\leq l+1\leq p-1,$ Theorem \ref{E-M} can be
written as%
\begin{align}
\sum_{n=1}^{ck}\chi\left(  n\right)  \overline{B}_{p,\chi}\left(  ny\right)
&  =\frac{\chi\left(  -1\right)  }{p+1}\sum\limits_{j=0}^{l}\binom{p+1}%
{j+1}\left(  -1\right)  ^{j+1}y^{j}\left\{  \overline{B}_{p-j,\chi}\left(
cky\right)  -B_{p-j,\chi}\right\}  B_{j+1,\overline{\chi}}\nonumber\\
&  \quad-\chi\left(  -1\right)  \left(  -y\right)  ^{l+1}\binom{p}{l+1}%
\int\limits_{0}^{ck}\overline{B}_{l+1,\overline{\chi}}\left(  u\right)
\overline{B}_{p-l-1,\chi}\left(  yu\right)  du. \label{23}%
\end{align}

I) Let $y=b/c,$ $\alpha=0$ and $\beta=ck.$ Then (\ref{23}) becomes
\begin{equation}
\sum_{n=1}^{ck}\chi\left(  n\right)  \overline{B}_{p,\chi}\left(  \frac{b}%
{c}n\right)  =-\chi\left(  -1\right)  \binom{p}{l+1}\left(  -\frac{b}%
{c}\right)  ^{l+1}\int\limits_{0}^{ck}\overline{B}_{l+1,\overline{\chi}%
}\left(  u\right)  \overline{B}_{p-l-1,\chi}\left(  \frac{b}{c}u\right)  du.
\label{24}%
\end{equation}

$\bullet$ If $b=c,$ combining (\ref{24}) and Lemma \ref{lek2} with $d=c=1$ we
deduce that%

\[
\int\limits_{0}^{k}\overline{B}_{r,\overline{\chi}}\left(  u\right)
\overline{B}_{m,\chi}\left(  u\right)  du=\left(  -1\right)  ^{r-1}\frac
{r!m!}{\left(  m+r\right)  !}\left(  k^{m+r}-1\right)  B_{m+r}%
\]
for $k$ as in Lemma \ref{lek2} (where $l+1=r,$ $p-r=m$).

\begin{remark}
The equation above stated as
\[
\int\limits_{0}^{k}\overline{B}_{r,\overline{\chi}}\left(  u\right)
\overline{B}_{m,\chi}\left(  u\right)  du=\left(  -1\right)  ^{m-1}\frac
{m!r!}{\left(  m+r\right)  !}k^{m+r}B_{m+r},\text{ }m,r\geq1
\]
in \cite[Proposition 6.6 ]{b5}.
\end{remark}

$\bullet$ If $b=ck,$ it follows from the fact $\sum_{n=0}^{ck}\chi\left(
n\right)  =0$ and (\ref{24}) that
\begin{equation}
\int\limits_{0}^{k}\overline{B}_{l+1,\overline{\chi}}\left(  u\right)
\overline{B}_{p-l-1,\chi}\left(  ku\right)  du=0.\nonumber
\end{equation}

$\bullet$ Now assume that $\left(  b,c\right)  =1$.\ Then, it follows from
Lemma \ref{lek2} and (\ref{24}) that
\[
\int\limits_{0}^{k}\overline{B}_{l+1,\overline{\chi}}\left(  cu\right)
\overline{B}_{p-l-1,\chi}\left(  bu\right)  du=\frac{\left(  -1\right)  ^{l}%
}{\binom{p}{l+1}}\frac{c^{l+1-p}}{b^{l+1}}\chi\left(  c\right)  \overline
{\chi}\left(  b\right)  \left(  k^{p}-1\right)  B_{p}%
\]
for $k$ as in Lemma \ref{lek2}.

$\bullet$ Let $\left(  b,c\right)  =q$ and put $c=qc_{1},$\ $b=qb_{1}.$\ From
the equation above, we have
\begin{align*}
\int\limits_{0}^{k}\overline{B}_{l+1,\overline{\chi}}\left(  cu\right)
\overline{B}_{p-l-1,\chi}\left(  bu\right)  du  &  =\int\limits_{0}%
^{k}\overline{B}_{l+1,\overline{\chi}}\left(  c_{1}u\right)  \overline
{B}_{p-l-1,\chi}\left(  b_{1}u\right)  du\\
&  =q^{p}\frac{\left(  -1\right)  ^{l}}{\binom{p}{l+1}}\frac{c^{l+1-p}%
}{b^{l+1}}\chi\left(  \frac{c}{q}\right)  \overline{\chi}\left(  \frac{b}%
{q}\right)  \left(  k^{p}-1\right)  B_{p}%
\end{align*}
for $k$ as in Lemma \ref{lek2} with replacing $c$ by $c_{1}$\ and $b$ by
$b_{1}.$

II) Let $k$ and $b$\ be odd and consider $y=b/2c$ with $\left(  b,c\right)
=1$. From (\ref{68}), for odd $b$ we have%
\begin{equation}
\overline{B}_{p-j,\chi}\left(  \frac{bk}{2}\right)  =\overline{B}_{p-j,\chi
}\left(  \frac{k}{2}\right)  =\left\{  2^{j+1-p}\chi\left(  2\right)
-1\right\}  B_{p-j,\chi}. \label{25}%
\end{equation}
Therefore, (\ref{23}) becomes%
\begin{align}
s_{4,p}\left(  b,c:\chi\right)   &  =\sum_{n=1}^{ck}\chi\left(  n\right)
\overline{B}_{p,\chi}\left(  \frac{b}{2c}n\right) \nonumber\\
&  =-2^{1-p}\frac{\chi\left(  -1\right)  }{p+1}\sum\limits_{m=p-l}^{p}%
\binom{p+1}{m}\left(  -\frac{b}{c}\right)  ^{p-m}\left\{  \chi\left(
2\right)  -2^{m}\right\}  B_{m,\chi}B_{p+1-m,\overline{\chi}}\nonumber\\
&  -\chi\left(  -1\right)  c\binom{p}{l+1}\left(  -\frac{b}{2c}\right)
^{l+1}\int\limits_{0}^{k}\overline{B}_{l+1,\overline{\chi}}\left(  cu\right)
\overline{B}_{p-l-1,\chi}\left(  \frac{b}{2}u\right)  du \label{26}%
\end{align}
by setting $j=p-m$.

For $l=0$\ we have the following integral representation:
\begin{equation}
\chi\left(  -1\right)  s_{4,p}\left(  b,c:\chi\right)  =\frac{pb}{2}%
\int\limits_{0}^{k}\overline{B}_{1,\overline{\chi}}\left(  cu\right)
\overline{B}_{p-1,\chi}\left(  \frac{b}{2}u\right)  du-\left\{  2^{1-p}%
\chi\left(  2\right)  -2\right\}  B_{p,\chi}B_{1,\overline{\chi}}. \label{26a}%
\end{equation}

$\bullet$ If $p$ is even, it is seen from (\ref{26}) and Proposition \ref{pr1}
that
\begin{align*}
&  c\binom{p}{l+1}\left(  -\frac{b}{2c}\right)  ^{l+1}\int\limits_{0}%
^{k}\overline{B}_{l+1,\overline{\chi}}\left(  cu\right)  \overline
{B}_{p-l-1,\chi}\left(  \frac{b}{2}u\right)  du\\
&  \ =-2^{-p}c^{1-p}\chi\left(  2c\right)  \overline{\chi}\left(  b\right)
\left(  k^{p}-1\right)  B_{p}-\frac{2^{1-p}}{p+1}\sum\limits_{m=p-l}^{p}%
\binom{p+1}{m}\left(  -\frac{b}{c}\right)  ^{p-m}\left\{  \chi\left(
2\right)  -2^{m}\right\}  B_{m,\chi}B_{p+1-m,\overline{\chi}}%
\end{align*}
for $k$ as in Lemma \ref{lek2}.

$\bullet$ Let $p$ be odd and put $l+1=p-1$ in (\ref{26}). Then,%
\begin{align}
&  \chi\left(  -1\right)  \left(  p+1\right)  b\left(  2c\right)  ^{p}%
s_{4,p}\left(  b,c:\chi\right) \nonumber\\
&  \ =2\sum\limits_{m=2}^{p}\binom{p+1}{m}\left(  -1\right)  ^{m}%
b^{p+1-m}c^{m}\left\{  \chi\left(  2\right)  -2^{m}\right\}  B_{m,\chi
}B_{p+1-m,\overline{\chi}}\nonumber\\
&  \qquad-2cb^{p}p\left(  p+1\right)  c\int\limits_{0}^{k}\overline
{B}_{p-1,\overline{\chi}}\left(  cu\right)  \overline{B}_{1,\chi}\left(
\frac{b}{2}u\right)  du. \label{29}%
\end{align}

III) Let $k$ and $c$ be odd and consider $y=2b/c,$ $\alpha=0,$ $\beta=ck/2$
with $\left(  b,c\right)  =1$. Then, from Theorem \ref{E-M} and Eq. (\ref{25})
we have
\begin{align}
&  \sum_{0\leq n\leq ck/2}\chi\left(  n\right)  \overline{B}_{p,\chi}\left(
\frac{2b}{c}n\right)  =\sum_{n=1}^{\left(  ck-1\right)  /2}\chi\left(
n\right)  \overline{B}_{p,\chi}\left(  \frac{2b}{c}n\right) \nonumber\\
&  \ =-\frac{\chi\left(  -1\right)  }{p+1}\sum\limits_{j=1}^{l+1}\binom
{p+1}{j}\left(  -\frac{b}{c}\right)  ^{j-1}\left\{  \overline{\chi}\left(
2\right)  -2^{j}\right\}  B_{j,\overline{\chi}}B_{p+1-j,\chi}\nonumber\\
&  \quad-\chi\left(  -1\right)  \left(  -\frac{2b}{c}\right)  ^{l+1}\binom
{p}{l+1}\frac{c}{2}\int\limits_{0}^{k}\overline{B}_{l+1,\overline{\chi}%
}\left(  \frac{c}{2}u\right)  \overline{B}_{p-l-1,\chi}\left(  bu\right)
du.\text{ \ \ } \label{27}%
\end{align}
Now consider the sum $s_{3,p}\left(  b,c:\chi\right)  .$%
\begin{align}
s_{3,p}\left(  b,c:\chi\right)   &  =\sum\limits_{n=1}^{ck}\left(  -1\right)
^{n}\chi\left(  n\right)  \overline{B}_{p,\chi}\left(  \frac{bn}{c}\right)
\nonumber\\
&  =2\chi\left(  2\right)  \sum\limits_{n=1}^{\frac{ck-1}{2}}\chi\left(
n\right)  \overline{B}_{p,\chi}\left(  \frac{2bn}{c}\right)  -\sum
\limits_{n=1}^{ck}\chi\left(  n\right)  \overline{B}_{p,\chi}\left(  \frac
{bn}{c}\right)  . \label{28}%
\end{align}

$\bullet$ If $p$ is even, then from (\ref{27}), Proposition \ref{pr1} and
Lemma \ref{lek2}
\begin{align}
&  c\left(  -\frac{2b}{c}\right)  ^{l+1}\binom{p}{l+1}\int\limits_{0}%
^{k}\overline{B}_{l+1,\overline{\chi}}\left(  \frac{c}{2}u\right)
\overline{B}_{p-l-1,\chi}\left(  bu\right)  du\nonumber\\
&  \ =-c^{1-p}\chi\left(  c\right)  \overline{\chi}\left(  2b\right)  \left(
k^{p}-1\right)  B_{p}-\frac{2}{p+1}\sum\limits_{j=1}^{l+1}\binom{p+1}%
{j}\left(  -\frac{b}{c}\right)  ^{j-1}\left\{  \overline{\chi}\left(
2\right)  -2^{j}\right\}  B_{j,\overline{\chi}}B_{p+1-j,\chi}\nonumber
\end{align}
for $k$ as in Lemma \ref{lek2}.

$\bullet$ Let $p$ be odd. Put $l+1=p-1$ in (\ref{27}). Then, (\ref{27}),
(\ref{28}) and Lemma \ref{lek2} yield
\begin{align}
\overline{\chi}\left(  -2\right)  \left(  p+1\right)  bc^{p}s_{3,p}\left(
b,c:\chi\right)   &  =2\sum\limits_{j=1}^{p-1}\binom{p+1}{j}\left(  -1\right)
^{j}b^{j}c^{p+1-j}\left\{  \overline{\chi}\left(  2\right)  -2^{j}\right\}
B_{j,\overline{\chi}}B_{p+1-j,\chi}\nonumber\\
&  \quad-2^{p-1}b^{p}c^{2}p\left(  p+1\right)  \int\limits_{0}^{k}\overline
{B}_{p-1,\overline{\chi}}\left(  \frac{c}{2}u\right)  \overline{B}_{1,\chi
}\left(  bu\right)  du. \label{30}%
\end{align}

Putting $l=0$\ in (\ref{27}) gives an integral representation for
$_{3,p}\left(  b,c:\chi\right)  $:\
\begin{equation}
\frac{1}{2}\overline{\chi}\left(  -2\right)  s_{3,p}\left(  b,c:\chi\right)
=bp\int\limits_{0}^{k}\overline{B}_{1,\overline{\chi}}\left(  \frac{c}%
{2}u\right)  \overline{B}_{p-1,\chi}\left(  bu\right)  du-\left\{
\overline{\chi}\left(  2\right)  -2\right\}  B_{1,\overline{\chi}}B_{p,\chi}.
\label{33}%
\end{equation}

Combining (\ref{30}) and (\ref{26a}) (or (\ref{29}) and (\ref{33})) we arrive
at the reciprocity formula
\begin{align*}
&  \chi\left(  -1\right)  \left(  p+1\right)  \left\{  c\left(  2b\right)
^{p}s_{4,p}\left(  c,b:\overline{\chi}\right)  +\overline{\chi}\left(
2\right)  bc^{p}s_{3,p}\left(  b,c:\chi\right)  \right\} \\
&  \ =2\sum\limits_{j=1}^{p}\binom{p+1}{j}\left(  -1\right)  ^{j}%
b^{j}c^{p+1-j}\left\{  \overline{\chi}\left(  2\right)  -2^{j}\right\}
B_{j,\overline{\chi}}B_{p+1-j,\chi}%
\end{align*}
given by Corollary \ref{ck2} without the restriction $c$\ or $b\equiv
0(mod$\ $k)$.

\end{document}